\documentclass[10pt]{amsart}
\usepackage{amsmath}
\usepackage{amsthm}
\usepackage{amssymb}
\usepackage{enumerate}
\usepackage[margin=1.2 in]{geometry}
\usepackage{hyperref}
\usepackage[all]{xy}
\usepackage{verbatim}
\usepackage{ifthen}
\usepackage{mathrsfs}
\usepackage{setspace}
\usepackage{tikz}
\usepackage{graphicx}
\usepackage{breqn}

\usepackage{bbm} 

\newtheorem{thm}{Theorem}[section]
\newtheorem{lemma}[thm]{Lemma}
\newtheorem{prop}[thm]{Proposition}

\newtheorem{defn}[thm]{Definition}

\theoremstyle{definition}

\newtheorem{question}[thm]{Question}
\newtheorem*{ack}{Acknowledgement}

\title{A metric sphere not a quasisphere but for which every weak tangent is Euclidean}
\author{Angela WU}
\address{Angela Wu\\Department of Mathematics\\
	University of California\\ Los Angeles\\ CA90095 \\USA} \email{aawu@math.ucla.edu}
\thanks{The author was partially supported by NSF grants DMS-1506099 and DMS-1266164}
\date{June 7, 2018}

\begin{document}
\newcommand{\floor}[1]{\lfloor #1 \rfloor}
\newcommand{\ceil}[1]{\lceil #1 \rceil}
\newcommand{\tr}{\operatorname{tr}}
\newcommand{\Tr}{\operatorname{Tr}}
\newcommand{\GL}{\operatorname{GL}}
\newcommand{\Card}{\operatorname{Card}}
\newcommand{\card}{\operatorname{card}}
\newcommand{\onto}[1][]{\ifthenelse{\equal{#1}{}}{\twoheadrightarrow}{\overset{#1}\twoheadrightarrow}}
\newcommand{\into}[1][]{\ifthenelse{\equal{#1}{}}{\hookrightarrow}{\overset{#1}\hookrightarrow}}
\newcommand{\dive}{\operatorname{div}}
\newcommand{\grad}{\nabla}
\newcommand{\coker}{\operatorname{coker}}
\newcommand{\wto}{\rightharpoonup}
\newcommand{\dd}{\mathrm{d}}
\newcommand{\Gal}{\operatorname{Gal}}
\newcommand{\rank}{\operatorname{rank}}
\newcommand{\Aut}{\operatorname{Aut}}
\newcommand{\Mor}{\operatorname{Mpr}}
\newcommand{\Hom}{\operatorname{Hom}}
\newcommand{\End}{\operatorname{End}}
\newcommand{\Hol}{\operatorname{Hol}}
\newcommand{\Var}{\operatorname{Var}}
\newcommand{\Cor}{\operatorname{Cor}}
\newcommand{\Cov}{\operatorname{Cov}}
\newcommand{\ind}{\operatorname{ind}}
\newcommand{\ch}{\operatorname{char}}
\newcommand{\vep}{\varepsilon}
\newcommand{\Image}{\operatorname{Im}}
\newcommand{\im}{\operatorname{im}}
\newcommand{\sgn}{\operatorname{sgn}}
\newcommand{\Id}{\operatorname{Id}}
\newcommand{\id}{\operatorname{id}}
\newcommand{\Ell}{\mathscr{L}}
\newcommand{\CH}{\what{\C}}
\newcommand{\R}{\mathbb{R}}
\newcommand{\C}{\mathbb{C}}
\newcommand{\E}{\mathbb{E}}
\newcommand{\Q}{\mathbb{Q}}
\newcommand{\D}{\mathbb{D}}
\newcommand{\Z}{\mathbb{Z}}
\newcommand{\N}{\mathbb{N}}
\newcommand{\F}{\mathbb{F}}
\newcommand{\bP}{\mathbb{P}}
\newcommand{\T}{\mathbb{T}}
\newcommand{\supp}{\operatorname{supp}}
\newcommand{\spn}{\operatorname{span}}
\renewcommand{\H}{\mathbb{H}}
\newcommand{\fR}{\mathfrak{R}}
\newcommand{\br}[1]{\left( #1 \right)} 
\newcommand{\abs}[1]{\left| #1 \right|} 
\newcommand{\Lpnorm}[2]{\| #1 \|_{L^{#2}}}
\newcommand{\norm}[1]{\left\lVert#1\right\rVert} 
\newcommand{\angleb}[1]{\left\langle #1 \right\rangle} 
\newcommand{\abr}[1]{\left\langle #1 \right\rangle}
\newcommand{\bc}[1]{\left\lbrace #1 \right\rbrace}
\newcommand{\M}{\mathfrak{M}} 
\newcommand{\fa}{\mathfrak{a}} 
\newcommand{\bs}[1][]{\ifthenelse{\equal{#1}{}}{\backslash}{\backslash\{#1\}}}
\newcommand{\ra}[1]{\overset{#1}\longrightarrow}
\newcommand{\actson}{\curvearrowright}
\newcommand{\cA}{\mathcal{A}}
\newcommand{\cB}{\mathcal{B}}
\newcommand{\cC}{\mathcal{C}}
\newcommand{\cD}{\mathcal{D}}
\newcommand{\cE}{\mathcal{E}}
\newcommand{\cF}{\mathcal{F}}
\newcommand{\cG}{\mathcal{G}}
\newcommand{\cH}{\mathcal{H}}
\newcommand{\cK}{\mathcal{K}}
\newcommand{\cL}{\mathcal{L}}
\newcommand{\cM}{\mathcal{M}}
\newcommand{\cN}{\mathcal{N}}
\newcommand{\cP}{\mathcal{P}}
\newcommand{\cS}{\mathcal{S}}
\newcommand{\cT}{\mathcal{T}}
\newcommand{\cU}{\mathcal{U}}
\newcommand{\cV}{\mathcal{V}}
\newcommand{\cW}{\mathcal{W}}
\newcommand{\cX}{\mathcal{X}}
\newcommand{\cY}{\mathcal{Y}}
\newcommand{\cZ}{\mathcal{Z}}
\newcommand{\fX}{\mathfrak{X}}
\newcommand{\Ob}{Ob}
\newcommand{\fm}{\mathfrak{m}} 
\newcommand{\fp}{\mathfrak{p}} 
\newcommand{\m}{\mathfrak{m}} 
\newcommand{\p}{\mathfrak{p}} 
\newcommand{\what}[1]{\widehat{#1}}
\newcommand{\wtil}[1]{\widetilde{#1}}
\newcommand{\wbar}[1]{\overline{#1}}
\newcommand{\ub}[1]{\overline{#1}} 
\newcommand{\lb}[1]{\underline{#1}} 
\renewcommand{\mod}{\operatorname{mod}}
\newcommand{\diam}{\operatorname{diam}}

\begin{abstract}
We show that for all $n \geq 2$, there exists a doubling linearly locally contractible metric space $X$ that is topologically a $n$-sphere such that every weak tangent is isometric to $\R^n$ but $X$ is not quasisymmetrically equivalent to the standard $n$-sphere. The same example shows that $2$-Ahlfors regularity in Theorem 1.1 of \cite{BK02} on quasisymmetric uniformization of metric $2$-spheres is optimal.
\end{abstract}

\maketitle

\section{Introduction}

Let $(X, d_X)$ and $(Y, d_Y)$ be two metric spaces. A homeomorphism $f: X \to Y$ is a quasisymmetry if there exists a homeomorphism $\eta:[0,\infty) \to [0,\infty)$ such that for all $x,y,z \in X$, with $x \neq z$, we have 
\begin{align*}
\frac{d_Y(f(x), f(y))}{d_Y(f(x), f(z))} \leq \eta\br{ \frac{d_X(x,y)}{d_X(x,z)}}.
\end{align*}
In this case we say that $X$ and $Y$ are quasisymmetrically equivalent. 

In \cite{Kinneberg17} Kinneberg characterized metric circles that are quasisymmetric to the standard circle in terms of weak tangents (defined in Section 4 below):
\begin{thm}[Kinneberg, \cite{Kinneberg17}]
A doubling metric circle $\mathcal{C}$ is quasisymmetrically equivalent to the standard circle $\mathbb{S}^1$ if and only if every weak tangent of $\mathcal{C}$ quasisymmetrically equivalent to the real line $\R$ based at $0$. 
\end{thm}
\noindent Here a metric space is said to be doubling if for every $R > 0$, every ball $B$ of radius $2R$ can be covered by $N$ balls of radius $R$, where $N$ is a positive integer that does not depend on the choice of the ball $B$. In this paper we prove that Kinneberg's result cannot be extended to higher dimensions:
\begin{thm}\label{Theorem}
For every $n \geq 2$, there exists a doubling, linearly locally contractible metric space $X$ that is topologically an $n$-sphere such that every weak tangent is isometric to $\R^n$ but $X$ is not quasisymmetrically equivalent to the standard $n$-sphere.
\end{thm}

When $n = 2$, one can compare our result with the following Theorem:
\begin{thm}[Bonk and Kleiner, \cite{BK02}]\label{Bonk_Kleiner}
Let $Z$ be an 2-Ahlfors regular metric space homeomorphic to $\mathbb{S}^2$. Then $Z$ is quasisymmetric to $\mathbb{S}^2$ if and only if $Z$ is linearly locally contractible.
\end{thm}
\noindent Our Theorem \ref{Theorem} shows that the conclusion of Theorem \ref{Bonk_Kleiner} is false if we replace $2$-Ahlfors regularity with $Q$-Ahlfors regularity for $Q >2$.

Our study is also related to the following theorem, proven in \cite{Lindquist17}:
\begin{thm}\label{Intro_Thm1}
Let $(Z,d)$ be a doubling metric space homeomorphic to $\mathbb{S}^2$. The following are equivalent:
\begin{enumerate}[(i)]
\item $(Z, d)$ is quasisymmetrically equivalent to the standard $2$-sphere.
\item For every $x \in Z$, there exists an open neighborhood $U$ of $x$ in $Z$ such that $U$ is quasisymmetrically equivalent to $\D$.
\end{enumerate}
\end{thm}
\noindent An alternative proof of Theorem \ref{Intro_Thm1} can be given using ideas in \cite{GW18}. Roughly speaking, Theorem \ref{Intro_Thm1} says local geometry properties promote to global property. Since weak tangents are local, one could ask the following question:
\begin{question}
Suppose $(Z,d)$ is doubling and linearly locally connected. Are the following two statements equivalent?
\begin{enumerate}[(i)]
\item $Z$ is quasisymmetrically equivalent to the standard $d$-sphere $\mathbb{S}^d$
\item Every weak tangent of $Z$ is quasisymmetrically equivalent to $\R^d$. 
\end{enumerate}
\end{question}
\noindent When $Z$ is a doubling and linearly locally connected metric sphere, statement (i) implies statement (ii). However, our construction shows that statement (ii) does not imply statement (i).

\begin{ack} The author would like to thank Mario Bonk and John Garnett for their support and many conversations. The author thanks Jeff Lindquist for helpful discussions and Wenbo Li for his comments and corrections.
\end{ack}

\section{The Modification of the Metric on $\R$}

We first construct a metric $\delta$ on $\R$ so that every weak tangent of $(\R, \delta)$ is isometric to $(\R, 0, d)$, where $d$ is the Euclidean metric on $\R$, and so that there are segments $[a_j, b_j]$ of $(\R, \delta)$ so that $\ell([a_j, b_j]) / \delta(a_j, b_j) \to +\infty$. Here, and in the remaining of the paper, $\ell(\gamma)$ denotes the length of the curve $\gamma$ with respect to the metric $\delta$. Once the metric $\delta$ is constructed, we show that for all $n \geq 2$, the weak tangents of the product space $(\R, \delta) \times \R^{n-1}$ are isometric to the Euclidean space $\R^n$, but because $\ell([a_j, b_j]) / \delta(a_j, b_j) \to +\infty$, the product space $(\R, \delta) \times \R^{n-1}$ cannot be quasisymmetric to the Euclidean space $\R^n$. 

To construct $(\R, \delta)$, we glue together segments $[a_j, b_j]$ that look like snowflake curves of Hausdorff dimension $(\alpha_j)^{-1}$ i.e. intervals $I$ equipped with the metric $d^{\alpha_j}$. Such snowflake curves will have infinite length, so we modify the metric on the segment $[a_j, b_j]$ to get a metric $d_n$ so that $d_n$ looks like the Euclidean metric when two points are close together, but it looks like $d^{\alpha_j}$ when the two points are far apart.

Note that our construction of $(\R, \delta)$ only modifies the geometry of a bounded subset of $\R$, and therefore we can embed our $n$-dimensional construction into a topological $n$-sphere. 

To begin the construction let $\alpha \in (0,1)$ and $c \in (0,1)$ and define
\[ \varphi_{\alpha, c}(x) = \begin{cases}
L(\alpha, c) x,& x \in [0,c] \\
\br{ \frac{x - c(1-\alpha)}{1 - c(1-\alpha)}}^\alpha,& x \in [c,1],
\end{cases} \]
where 
\[ L(\alpha, c) = \frac{1}{c} \br{\frac{c\alpha}{1 - c(1-\alpha)}}^\alpha. \]
The function $\varphi_{\alpha, c}$ is the only function that has the following properties:
\begin{enumerate}
\item $\varphi_{\alpha, c}(0) = 0$,
\item $\varphi_{\alpha, c}(1) = 1$,
\item $\varphi_{\alpha, c}$ is linear on $[0,c]$,
\item There exists $a,b \in \R$, with $a \neq 0$, such that $\varphi_{\alpha ,c}(x) = (ax - b)^\alpha$ when $x \in [c,1]$,
\item $\varphi_{\alpha, c}$ is continous and differentiable at $c$.
\end{enumerate}
In addition to the above properties that uniquely define $\varphi_{\alpha, c}$, we observe that $\varphi_{\alpha, c}$ has the following extra properties:
\begin{enumerate}
\item $\varphi_{\alpha, c} :[0,1] \to [0,1]$ is a homeomorphism,
\item $\varphi_{\alpha, c}$ is concave on $[0,1]$.
\end{enumerate}

\begin{lemma}\label{lemma_sup_multiplicative}
For any $\alpha \in (0,1)$, $c \in (0,1]$, and $a, b \in [0,1]$, we have 
\[ \varphi_{\alpha, c}(ab) \varphi_{\alpha, c}(a) \varphi_{\alpha, c}(b). \]
\end{lemma}
\begin{proof}
The function $L(\alpha, c)$ is decreasing in $c$ to $1$, therefore if $0 < c_1 \leq c_2 \leq 1$, then
\[ \varphi_{\alpha, c_1} \geq \varphi_{\alpha, c_2}. \]
Let 
\[ \varphi_\ast(x) = \frac{\varphi_{\alpha, c}(ax)}{\varphi_{\alpha, c}(a)}.  \]
The function $\varphi_\ast$ meets condition (1)-(5) listed above for some $c_\ast \in [c, 1]$, therefore $\varphi_{\ast} = \varphi_{\alpha, c_\ast}$. This implies that $\varphi_\ast \leq \varphi_{\alpha, c}(x)$ for all $x \in [0,1]$. Taking $x = b$, we have 
\[ \varphi_{\alpha, c}(ab) \leq \varphi_{\alpha, c}(a)\varphi_{\alpha, c}(b). \]
\end{proof}

\begin{lemma}\label{lemma_rough_isometries_1}
For any fixed $\alpha \in (0,1)$, for any $c \in (0,1)$, and whenever $0 \leq t \leq x \leq 1$, we have
\[ 0 \leq \varphi_{\alpha, c}(t) + \varphi_{\alpha, c}(x-t) -\varphi_{\alpha, c}(x) \leq \br{ 2 - 2^{\alpha}}\varphi_{\alpha, c}(x). \]
\end{lemma}
\begin{proof}
When $x\leq c$, for all $t \in [0,x]$, we have $\varphi_{\alpha, c}(t) + \varphi_{\alpha, c}(x-t) -\varphi_{\alpha, c}(x) = 0$. When $x \geq 2c$ is fixed, $\varphi_{\alpha, c}(t) + \varphi_{\alpha, c}(x-t) -\varphi_{\alpha, c}(x)$ is maximized when
\[ \varphi_{\alpha, c}'(t) - \varphi_{\alpha, c}'(x-t) = 0, \]
which is possible only if $x- t = t$ i.e. $2t = x$. Then we have
\begin{align*}
\varphi_{\alpha, c}(t) + \varphi_{\alpha, c}(x-t) -\varphi_{\alpha, c}(x)
&\leq 2\varphi_{\alpha, c}(x/2) - \varphi_{\alpha, c}(x)  \\
&\leq 2\varphi_{\alpha, c}(x/2) - 2^{\alpha}\varphi_{\alpha, c}(x/2) \\
&= (2 - 2^{\alpha})\varphi_{\alpha, c}(x/2).
\end{align*}
When $c < x < 2c$, we have 
\[ \varphi_{\alpha, c}(t) + \varphi_{\alpha, c}(x-t) \leq L(\alpha, c)x = \frac{x}{c} \varphi_{\alpha, c}(c). \]
By the concavity of $\varphi_{\alpha, c}$, we have 
\[ \frac{\varphi_{\alpha, c}(x)}{x} \geq \frac{\varphi_{\alpha, c}(2c)}{2c}. \]
Therefore
\begin{align*}
\varphi_{\alpha, c}(t) + \varphi_{\alpha, c}(x-t) -  \varphi_{\alpha, c}(x)
&\leq \frac{x}{c} \varphi_{\alpha, c}(c)-  \frac{x \varphi_{\alpha, c}(2c)}{2c} 
= \frac{x}{2c} \br{ 2 \varphi_{\alpha, c}(c) -\varphi_{\alpha, c}(2c)}.
\end{align*}
But
\[ 2 \varphi_{\alpha, c}(c) -\varphi_{\alpha, c}(2c) \leq  (2 - 2^{\alpha})\varphi_{\alpha, c}(c) = (2 - 2^{\alpha})cL(\alpha, c). \]
Therefore
\begin{align*}
\varphi_{\alpha, c}(t) + \varphi_{\alpha, c}(x-t) -  \varphi_{\alpha, c}(x)
\leq \frac{x}{2c} (2 - 2^{\alpha})cL(\alpha, c) 
= (2 - 2^{\alpha}) \varphi_{\alpha, c}(x/2).
\end{align*}
In any case, we have 
\[ 0 \leq \varphi_{\alpha, c}(t) + \varphi_{\alpha, c}(x-t) -  \varphi_{\alpha, c}(x) \leq (2 - 2^{\alpha}) \varphi_{\alpha, c}(x/2) \leq (2 - 2^{\alpha}) \varphi_{\alpha, c}(x). \]
\end{proof}

\section{The One Dimensional Construction}

For any $\alpha \in (0,1)$, we have
\begin{equation}\label{equation_1}
\lim_{c \to 0^+} L(\alpha, c) = +\infty.
\end{equation}
Let $\alpha_n$ be an increasing sequence in $(0,1)$ such that $\lim_{n \to \infty} \alpha_n = 1$. By (\ref{equation_1}), we can choose $c_n \in (0,1)$ such that $c_n \to 0$, and
\[ \lim_{n \to \infty} L(\alpha_n, c_n) = +\infty. \]
Let $\varphi_n = \varphi_{\alpha_n, c_n}$. Choose a sequence $s_n$ so that for all $n \in \N$, $s_n < 2\br{\frac{1}{n} - \frac{1}{n+1}}$, such that $s_n L(\alpha_n, c_n)$ is decreasing and $\sum_{n \in\N} s_n L(\alpha_n, c_n) < \infty$. 

For all $n \in \N$, let $a_n = \frac{1}{n} -s_n, b_n = \frac{1}{n}$. 
Let $I_n = [a_n, b_n]$, and equip $I_n$ with the metric $\delta_n = s_n\varphi_n \circ (s_n^{-1} d)$, where $d$ is the usual Euclidean metric on $I_n$. Note that 
\begin{enumerate}
\item The distance between the two endpoints of $I_n$ is $\delta_n(a_n, b_n) = s_n$. 
\item $I_n$ is rectifiable and the length of $I_n$ is $\ell(I_n) = s_nL(\alpha_n, c_n)$.
\end{enumerate}
We construct a metric $\delta$ on $\R$ so that if $x \leq y$, 
\begin{enumerate}
\item $\delta = \delta_n$ when restricted to $I_n \times I_n$,
\item $\delta(x,y) = d(x,y)$ when $x,y \in \R \bs \bigcup_{i \geq } I_i$,
\item $\delta(x,y) = d(x, a_n) + \delta_n(a_n, y)$ when $x \in \R \bs \bigcup_{i \in \N} I_i$ and $y \in I_n$, 
\item $\delta(x,y) = \delta_n(x, b_n) + d(b_n, y)$ when $x \in I_n$ and $y \in  \R \bs \bigcup_{i \in \N} I_i$, and
\item $\delta(x,y) = \delta_n(x, b_n) + d(b_n, a_m) + d_m(a_m, y)$ when $x \in I_n$, $y \in I_m$.
\end{enumerate}

\section{The Weak Tangents of $(\R, \delta)$}

A \textit{pointed metric space} is a a triplet $(X, x_0, d_X)$, where $(X, d_X)$ is a metric space and $x_0$ is a point in $X$. Let $\vep > 0$. A map $\psi: (X, x_0, d_X) \to (Y, y_0, d_Y)$ between two pointed metric spaces $(X, x_0, d_X)$ and $(Y, y_0, d_Y)$ is a \textit{$\vep$-rough isometry} if
\begin{enumerate}
\item $\psi(x_0) = y_0$,
\item $d_Y(\psi(X), Y) \leq \vep$, and
\item for all $x_1, x_2 \in X$, 
\[\abs{ d_Y(\psi(x_1), \psi(x_2)) - d_X(x_1, x_2)} \leq \vep. \]
\end{enumerate}
Note that a $\vep$-rough isometry may not be continuous. The \textit{pointed Gromov-Hausdorff distance} between 2 pointed metric spaces, denoted $d_{GH}\br{(X, x_0, d_X), (Y, y_0, d_Y)}$, is defined as the infimum of $\vep$ for which we can find a $\vep$-rough isometry $\psi_\delta: (X, x_0, d_X) \to (Y, y_0, d_Y)$.

A pointed metric space $(T, p, d_T)$ is called a \textit{weak tangent} of another metric space $(X, d_X)$ if there exist points $x_n \in X$ and positive integers $\lambda_n \to +\infty$ such that $(X, x_n, \lambda_n d_X)$ converges to $(T, p, d_T)$ in pointed Gromov-Hausdorff sense, i.e. for all $R > 0$, and for all $\vep > 0$ there exists $N > 0$ such that for all $n \geq N$, 
\begin{align*}
d_{GH}\br{\wbar{B}_{\lambda_n d_X}(x_n, R + \vep), \wbar{B}_{d_T}(p, R)} \leq \vep. 
\end{align*}
In particular, we get $(X, x_n, \lambda_n d_X) \to (T, p, d_T)$ if for all $R > 0$, and for all $\vep > 0$ there exists $N > 0$ such that for all $n \geq N$, 
\begin{align*}
d_{GH}\br{\wbar{B}_{\lambda_n d_X}(x_n, R), \wbar{B}_{d_T}(p, R)} \leq \vep. 
\end{align*}
Our notion of pointed Gromov-Hausdorff convergence is adopted from \cite[Definition 11.3.1]{The_Purple_Book}. Also see \cite{Burago} for detailed discussion on Gromov-Hausdorff distance and Gromov-Hausdorff convergence. 

Let M be a set of separable, uniformly doubling, and uniformly linearly locally contractible pointed metric spaces. The pointed Gromov Hausdorff convergence on the set $M$ induces a topology on $M$ that is metrizable\cite{LeDonne}. In particular, if a sequence of pointed metric space converges to a Gromov-Hausdorff limit, then the limit is unique. 

The goal of this section is to prove the following proposition that describe all the weak tangents of $(\R, \delta)$.
\begin{prop}\label{Prop_1}
For all $a_n \in \R$, and for all positive integers $\lambda_n \to +\infty$, $(\R, a_n, \lambda_n \delta)$ converges in pointed Gromov-Hausdorff sense to $(\R, 0, d)$. 
\end{prop}
Note that Proposition \ref{Prop_1} guarentees the existence of weak tangents of $(\R, \delta)$. To prove the above proposition we will use the following three lemmas:

\begin{lemma}\label{lemma_rough_isometries_2}
Suppose $x, y, z \in \R$ are three points so that $x \leq y \leq z$. Suppose $N = \inf\{n \in \N: \{x,y,z\} \cap I_N \neq \emptyset\} < \infty$. Then 
\[ 0 \leq \delta(x,y) + \delta(y,z) - \delta(z,x) \leq \br{2 - 2^{\alpha_N}} \min\{ s_N, \delta(x,y) \} \]
\end{lemma}
\begin{proof}
If $N = +\infty$, then $\delta(x,y) + \delta(y,z) - \delta(z,x) = 0$. Otherwise, let 
\begin{align*}
a &= \sup\{ a_n \leq y: n \in \N\} \vee \sup\{b_n \leq y: n \in \N\} \vee x \\
b &= \inf\{ a_n \geq y: n \in \N\} \wedge \inf\{b_n \geq y: n \in \N\} \wedge z
\end{align*}
We have $x \leq a \leq y \leq b \leq z$. By definition of $\delta$, we have 
\begin{align*}
\delta(x,y) + \delta(y,z) - \delta(x,z) 
&= \br{ \delta(x,a) + \delta(a,y)} + \br{ \delta(y,b) + \delta(b,z)} \\
&\hspace{30pt} - \br{ \delta(x,a) + \delta(a,b) + \delta(b,z)} \\
&= \delta(a,y) + \delta(y,b) -\delta(a,b).
\end{align*}
Bu our choice of $a$ and $b$, either $a,y,b \in I_n$ for some $n \geq N$, or $(a, b) \cap \bigcup_{n \in \N} I_n = \emptyset$. In the latter case, we have 
\[ \delta(x,y) + \delta(y,z) - \delta(z,x) = \delta(a,y) + \delta(y,b) -\delta(a,b) = 0. \]
In the former case, Lemma \ref{lemma_rough_isometries_1} gives
\begin{align*}
0 &\leq \delta(x,y) + \delta(y,z) - \delta(z,x) 
= \delta(a,y) + \delta(y,b) -\delta(a,b)  \\
&\leq \br{2 - 2^{\alpha_n}} \delta(a,b) 
= \br{2 - 2^{\alpha_N}} \min\{ s_N, \delta(x,y) \}. 
\end{align*}
Since $\alpha_N \leq \alpha_n \leq 1$, we have our desired conclusion.
\end{proof}

\begin{lemma}\label{lemma_rough_isometries_3'}
Let $p \in \R$ and $r > 0$ be arbitrary. Suppose $N = \inf\{n \in \N: \wbar{B}_{\delta}(p,r)\cap I_n \neq \emptyset\} < \infty$. Let $a = \inf\{x \in \R: \delta(x,p) \leq r\}$ and $b = \sup\{x \in \R: \delta(x,p) \leq r\}$. The map 
\begin{align*}
\psi: ([a,b], p, \delta) \to ([-r, r],0,d)\\
\psi(x) = \begin{cases}
-\delta(p,x) ,& x \leq p \\
\delta(p,x),& x > p.
\end{cases}
\end{align*}
is a $\vep$-rough isometry, where $\vep = \br{2 - 2^{\alpha_N}} \min\{ s_N, 2r\}$.
\end{lemma}
\begin{proof}
Since $\psi$ is surjective and fixes $p$, it remains to check that for all $x,y \in [a,b]$, we have 
\[ \br{d(\psi(x), \psi(y)) - \delta(x,y)} \leq \vep. \]
Suppose $x \leq y$. If $x \leq y \leq p$, then 
\begin{align*}
\abs{d(\psi(x), \psi(y)) - \delta(x,y)} 
&= \abs{ \abs{ \delta(p,x) - \delta(p,y)} - \delta(x,y)} \\
&= \abs{\delta(p,y) - \delta(p, x) - \delta(x,y)} \\
&\leq \br{2 - 2^{\alpha_N}} \min\{ s_N, \delta(p,x)\}  \\
&= \br{2 - 2^{\alpha_N}} \min\{ s_N, 2r\} .
\end{align*}
The second to last inequality is a consequence of Lemma \ref{lemma_rough_isometries_2}. If $x \leq p \leq y$, then 
\begin{align*}
\abs{d(\psi(x), \psi(y)) - \delta(x,y)} 
&= \abs{ \abs{ \delta(p,x) - \delta(p,y)} - \delta(x,y)} \\
&= \abs{\delta(p,y) + \delta(p, x) - \delta(x,y)} \\
&\leq \br{2 - 2^{\alpha_N}} \min\{ s_N, \delta(x,y)\}  \\
&= \br{2 - 2^{\alpha_N}} \min\{ s_N, 2r\}.
\end{align*}
If $p \leq x \leq y$, then following a similar arguement as when $x \leq y \leq p$, we get
\begin{align*}
\abs{d(\psi(x), \psi(y)) - \delta(x,y)}
\leq \br{2 - 2^{\alpha_N}} \min\{ s_N, 2r\}
\end{align*}
This verifies that $\psi$ is a $\vep$-rough isometry.
\end{proof}

\begin{lemma}\label{lemma_rough_isometries_3}
For $r \in (0,1)$, we have 
\[ \sup_{p \in \R} d_{GH}((\wbar{B}_{\delta}(p,r), p, \delta),(\wbar{B}_{d}(0,r), 0, d)) = o(r) \]
as $r \to 0$.
\end{lemma}
\begin{proof}
Let $p \in \R$ and $r > 0$ be arbitrary. Let $N = \inf\{n \in \N: \wbar{B}_{\delta}(p,r)\cap I_n \neq \emptyset\}$. If $N = +\infty$, then $\delta = d$ on $\wbar{B}_{\delta}(p,r)$. We have 
\[ d_{GH}((\wbar{B}_{\delta}(p,r), p, \delta),(\wbar{B}_{d}(0,r), 0, d)) = 0. \]
If $N < \infty$, we consider 2 cases:

Case 1: $r \geq \frac{s_N}{2}\varphi_N(c_N)$. As $r \to 0$, $N \to +\infty$, therefore $\br{2 - 2^{\alpha_N}} \to 0$. In this case Lemma \ref{lemma_rough_isometries_3'} implies
\begin{align*}
d_{GH}((\wbar{B}_{\delta}(p,r), p, \delta),(\wbar{B}_{d}(0,r), 0, d)) \leq 2\br{2 - 2^{\alpha_N}} r = o(r).
\end{align*} 

Case 2: $r < \frac{s_N}{2}\varphi_N(c_N)$. In this case, $\delta$ is a length metric on $[a,b]$, and $\psi$ is an isometry between two length spaces. We have 
\[ d_{GH}((\wbar{B}_{\delta}(p,r), p, \delta),(\wbar{B}_{d}(0,r), 0, d)) = 0. \]
\end{proof}

\begin{proof}[Proof of Proposition \ref{Prop_1}]
Let $\{a_n\}_{n \in \N}$ be a sequence in $\R$ and $\{\lambda_n\}_{n \in \N}$ be a sequence of positive numbers that diverges to $+\infty$. By Lemma \ref{lemma_rough_isometries_3}, 
\begin{align*}
&d_{GH}\br{(\wbar{B}_{\lambda_n \delta}(a_n, R ), a_n, \lambda_n \delta),(\wbar{B}_{\lambda_n d}(0, R), 0, \lambda_n d)} \\
&\hspace{30pt}= \lambda_n d_{GH}\br{(\wbar{B}_{\delta}(a_n, \lambda_n^{-1}R ), a_n, \delta),( \wbar{B}_{d}(0, \lambda_n^{-1}R), 0, d)} = \lambda_n o(\lambda_n^{-1}R).
\end{align*}
But $(\wbar{B}_{\lambda_n d}(0, R), 0, \lambda_n d) = \wbar{B}_{d}(0, R), 0, d))$ by symmetry of $\R$. As $n \to \infty$, $\lambda_n^{-1}R \to 0$. We have
\[ (\wbar{B}_{\lambda_n \delta}(a_n, R), a_n, \delta ) \to \wbar{B}_{d}(0, R), 0, d). \]
This is true for all $R > 0$. We conclude that $(\R, a_n, \lambda_n \delta) \to (\R, 0, d)$.
\end{proof}

\section{Linear Local Contractibility and Assouad Dimension of $(\R, \delta)$}

In this section we establish two properties of the space $(\R, \delta)$. These properties often appear in the study of quasisymmetry classes of metric spheres. Both properties are discussed in detail in \cite{Heinonen}.

\begin{defn}
Let $C > 1$ be a constant. A metric space is \textit{$C$-linearly locally contractible} if every small ball is contractible inside a ball whose radius is $C$ times larger.  A metric space is \textit{linearly locally contractible} if it is $C$-linearly locally contractible for some $C > 0$.
\end{defn}

\begin{defn}
Let $N > 0$. A metric space is \textit{$N$-doubling} if for all $R > 0$, every open ball of $2R$ can be covered by $N$ balls of radius $R$. A metric space is \textit{doubling} if it is $N$-doubling for some $N > 0$.
\end{defn}

Both doubling and linear local contractibility are preserved under quasisymmetry. The Euclidean spaces like $\R^n$ or $\mathbb{S}^n$ are doubling and linearly locally contractible. The doubling property also ensures the existence of weak tangents.

\begin{prop}\label{Prop_3}
The space $(\R, \delta)$ is $1$-linearly locally contractible.
\end{prop}
\begin{proof}
Any open ball $B$ in $(\R, \delta)$ is an open interval $(a,b)$. Denote $p$ the center of $B$ (in $(\R, \delta)$.) Note that the map $x \mapsto \delta(p, x)$ is increasing on $\{x \in \R: x \geq p$, and decreasing on $\{x \in \R: x \leq p\}$. Therefore the map $H(x,t) = tx + (1-t)p$ is a homotopy of $(a,b)$ to $\{p\}$ in $B$. This proves that $(\R, \delta)$ is $1$-linearly locally contractible.
\end{proof}

\begin{prop}\label{Prop_2}
The space $(\R, \delta)$ is doubling.
\end{prop}

If a metric space $X$ is doubling, then there exists $\beta > 0$ and $C > 0$ such that for all $\vep \in (0, 1/2)$ and $r > 0$, any set of diameter $r$ in $X$ can be covered by at most $C \vep^{-\beta}$ subsets of diameter at most $\vep r$. The function $\vep \mapsto C \vep^{-\beta}$ is called the \textit{covering function} of $X$. The Assouad dimension of $X$ is defined to be the infimum of all $\beta$ so that a covering function of the form $\vep \mapsto C \vep^{-\beta}$ of $X$ exists. Conversely, any metric space of finite Assouad dimension is doubling. Proposition \ref{Prop_2} will follow from the stronger proposition below. 

\begin{lemma}
For each $n \in \N$, the function $f_n(\vep) = 2\vep^{-\alpha_n^{-1}}$ is a covering function of $(I_n, \delta)$.
\end{lemma}
\begin{proof}
Every subinterval of $I_n$ of $\delta$-diameter $r \in [0, s_n]$ has $d$-diameter $s_n\varphi_n^{-1}(s_n^{-1}r)$. Thus our goal is to show that for every $r \in [0, s_n]$, and every $\vep \in (0, 1/2)$, every subset of $I_n$ of $d$-diameter $s_n\varphi_n^{-1}(s_n^{-1}r)$ can be covered by no more than $(\vep^{-\alpha_n}+1)$-many subintervals of $I_n$ of $d$-diameter at most $s_n\varphi_n^{-1}(s_n^{-1}\vep r)$. The number of subintervals we need can be bounded from above by
\begin{align*}
\frac{s_n\varphi_n^{-1}(s_n^{-1}r)}{s_n\varphi_n^{-1}(s_n^{-1}\vep r)} + 1
\leq \sup_{y \in (0,1]} \frac{\varphi_n^{-1}(y)}{\varphi_n^{-1}(\vep y)} + 1
= \sup_{y \in [\varphi_n(c_n),1]} \frac{\varphi_n^{-1}(y)}{\varphi_n^{-1}(\vep y)} + 1.
\end{align*}
We claim that the last supremum is attained when $y = 1$. This is equivalent to 
\begin{equation}\label{equation_2}
\frac{\varphi_n(\varphi_n^{-1}(\vep y))}{\varphi_n(\varphi_n^{-1}(y))} \leq \varphi_n\br{ \frac{\varphi_n^{-1}(\vep y)}{\varphi_n^{-1}(y)}}
\end{equation}
Inequality (\ref{equation_2}) is true by Lemma \ref{lemma_sup_multiplicative}. 

Suppose $\varphi_n(x_0) = \vep$. When $\vep > \varphi_n(c_n)$, we have $x_0 > c_n$, and
\begin{align*}
\vep^{\alpha_n^{-1}} = \frac{x_0 - c(1-\alpha_n)}{1-c(1-\alpha_n)} \leq x_0.
\end{align*}
When $0 < \vep < \varphi_n(c_n)$, we have $x_0 < c_n$ and $\vep = \varphi_n(x_0) = \frac{x_0}{c} \varphi_n(c_n)$. Since $ \varphi_n(c_n)^{\alpha_n^{-1}} \leq c$, we have 
\begin{align*}
\frac{1}{x_0} 
= \frac{\varphi(c)}{c\vep}
= \vep^{-\alpha_n^{-1}} \br{ \frac{\vep}{\varphi(c)}}^{\alpha_n^{-1}-1} \frac{\varphi(c)^{\alpha_n^{-1}}}{c} \leq \vep^{-\alpha_n^{-1}}. 
\end{align*}
In any case, we can take the covering function of $I_n$ to be
\begin{align*}
\vep^{-\alpha_n^{-1}}+1 \leq 2\vep^{-\alpha_n^{-1}}. 
\end{align*}
\end{proof}

\begin{prop}\label{Prop_4}
The Assouad dimension of $(\R, \delta)$ is $1$.
\end{prop}
\begin{proof} 
Let $\beta > 1$ be arbitrary. There exists $N \in \N$ such that when $n \geq N$, $\alpha^{-1} <\beta$. Let $C = \max_{n < N}\{2^{\alpha_n^{-1} - \beta}\} \geq 1$. Then the function $\vep \mapsto 2C \vep^{-\beta}$, where $\vep \in (0, 1/2]$, is a covering function of $(I_n, \delta)$ for all $n \in \N$. Thus $\vep \mapsto 4C \vep^{-\beta}$ is a covering function of $(\R, \delta)$. 

Since $(-\infty, 0)$ has Assouad dimension $1$, the proposition follows.
\end{proof}

\section{Higher Dimension Construction}

Let $d \geq 2$. We will denote by $d_{Euclid}$ the Euclidean metric on $\R^d$. Let 
\[ X_d = \br{\R \times \R^{d-1}, \sqrt{\delta^2 + d_{Euclid}^2}}\]
be the product of $(\R, \delta)$ and $(\R^{d-1}, d_{Euclid})$. Write $\rho_n = \sqrt{\delta^2 + d_{Euclid}^2}$. Here are some facts about $X_d$.

\begin{prop}\label{HD_Prop_1}
\begin{enumerate}[(a)]
\item Every weak tangent of $X_d$ is isometric to $(\R^d, 0, d_{Euclid})$. 
\item $X_d$ is doubling and linearly locally contractible.
\end{enumerate}
\end{prop}
\begin{proof}
\begin{enumerate}[(a)]
\item Every weak tangent of $X_d$ is of the form $(T \times \R^{n-1}, (x, 0), d_T \times d_{Euclid})$, where $(T, x, d_T)$ is a weak tangent of $(\R, \delta)$. 
\item Recall that $(\R, \delta)$ is doubling (Proposition \ref{Prop_2}). $X_d$ is product of doubling metric space, hence doubling. By Proposition \ref{Prop_3}, $(\R, \delta)$ is $C$-linearly localy contractible for some $C > 1$. Let $x = (x_1, x_2)$ be any point in $\R \times \R^{d-1}$, and $r >0$ be arbitrary. The ball $B(x,r)$ in $X_d$ can first be contracted to $\{x_1\} \times B(x_2, r)$ within a $B(x, Cr)$, which can then be contracted to the point $\{x_1, x_2\}$. 
\end{enumerate}
\end{proof}

Every finite segment in $(\R, \delta)$ is rectifiable. Let $\mu_1$ be a the measure on $(\R, \delta)$ given by length. For $d \geq 2$, let $\mu_d$ be the product measure $\mu_1 \times \lambda_{d-1}$ on $X_d$, where $\lambda_{d-1}$ is the $(d-1)$-dimensional Lebesgue measure on $\R^{d-1}$. 

In the remaining of this section we show that $X_d$ is not quasisymmetrically equivalent to $\R^d$. To do that we consider a geometric quantity that is roughly preserved under quasisymmetry called modulus. Given a family $\Gamma$ of curves in a measured metric space $(X, d_X, \mu)$, we say that a function $\rho:X \to [0,\infty)$ is admissible if for all $\gamma \in \Gamma$,
\begin{align*}
\int_\gamma \rho(x) \,ds \geq 1.
\end{align*}
Let $Q > 0$. We define the \textit{$Q$-modulus} of $\Gamma$ as
\begin{align*}
\mod_Q(\Gamma) = \inf\{ \int_X \rho^Q \,d\mu : \rho \text{ admissible }\}.
\end{align*}
Let $E, F \subset X$ be two disjoint nondegerate continua in $X$. Let $\Gamma_{E,F}$ to be the collection of all rectifiable curves joining $E$ and $F$. We write $\mod_Q(E,F) = \mod_Q(\Gamma_{E,F})$. 

Moduli behave nicely under quasisymmetry, as illustrated by the following theorem.
\begin{thm}[Tyson, \cite{Tyson98}]\label{Tyson_theorem}
Let $X, Y$ be locally compact, connected, $Q$-Ahlfors regular metric spaces, where $Q > 1$, and $f:X \to Y$ be a quasisymmetric homeomorphism. Then there exists $C > 1$ such that for all curve family $\Gamma \subset X$, we have 
\[ \frac{1}{C} \mod_Q(\Gamma) \leq \mod_Q(f(\Gamma)) \leq C\mod_Q(\Gamma). \]
\end{thm}
See the next section for the definition of $Q$-Ahlfors regularity. 

We can now prove that $X_d$ is not quasisymmetrically equivalent to $\R^d$. The idea is that if the two spaces are quasisymmetrically equivalent, then for any curve family $\Gamma$, $\Gamma$ and $f(\Gamma)$ should have comparable moduli. We know that $\R^d$ has the property that any disjoint nondegerate continua $E, F$ in $\R^d$ satisfy
\begin{equation}\label{modulus}
\mod_d(E,F) \leq \phi\br{ \frac{d(E,F)}{\diam(E) \wedge \diam(F)}}.
\end{equation}
for some non-increasing function $\phi:[0,\infty) \to (0,\infty)$. However, Proposition \ref{HD_Prop_3} shows that some sequence disjoint nondegerate continua $E_n, F_n$ in $X_d$ do not behaviour as in (\ref{modulus}). The only problem is that $X_d$ is not $d$-Ahlfors regular, so we cannot apply Theorem \ref{Tyson_theorem} directly. In Proposition \ref{HD_Prop_2}, however, we will show that the inequality
\[  \mod_d(E_n,F_n) \leq C \mod_d(f(E_n), f(F_n)) \]
holds for some $C$ independent of $C$. 

From now on, we will denote $\Delta(E,F) = \frac{d(E,F)}{\diam(E) \wedge \diam(F)}$.

\begin{prop}\label{HD_Prop_3}
There exists $E_n, F_n \subset X_d$ such that $\Delta(E_n, F_n) = \frac{d(E_n, F_n)}{\diam E_n \wedge \diam F_n} = 1$ and $\mod_d(E_n, F_n) \to +\infty$. 
\end{prop}
\begin{proof}
Take 
$E_n = I_n \times \{0\} \times [0,s_n]^{d-2}, 
 F_n = I_n \times \{s_n\} \times [0,s_n]^{d-2}$. 
Then $d(E_n, F_n) = s_n$, and $\diam E_n = \diam F_n = s_n$, therefore 
\[ \Delta(E_n, F_n) = \frac{d(E_n, F_n)}{\diam E_n \wedge \diam F_n} = 1. \]
For each $n \in \N$, and for $x \in I_n$ and $(v_2, \ldots, v_{d-1}) \in [0,s_n]^{d-2}$, let $\gamma_{x, v_2, \ldots, v_{d-1}}$ be the path
\[ t \mapsto \br{x, t, v_2, \ldots, v_{d-1}}, t \in [0,s_n]. \]
Let
\[ \Gamma_n = \{ \gamma_{x, v_2, \ldots, v_{d-1}} : x \in I_n, (v_2, \ldots, v_{d-1}) \in [0,s_n]^{d-2}\}\]
be the family of straight lines joining $E_n$ and $F_n$ that meet $E_n$ orthogonally. Then
\[ \mod_d(E_n, F_n) \geq \mod_d(\Gamma_n). \]

Let $\rho: X_d \to \R_{\geq 0}$ be an admissible function for $\Gamma_n$. This means for all $\gamma_{x, v_2, v_3, \ldots, v_{d-1}} \in \Gamma_n$, we have 
\[ \int_{\gamma_{x, v_2, v_3, \ldots, v_{d-1}}} \rho(t) \,dt 
= \int_0^{s_n} \rho(x,t,v_2, \ldots, v_{d-1}) \,dt \geq 1. \]
Integrating over $I_n \times [0,s_n]^{d-2}$ and applying Fubini's theorem, we get
\begin{align*} 
\int_{I_n \times [0,s_n]^{d-2}} \,d\mu_1 \times \lambda_{d-2}
&\leq \int_{I_n \times [0,s_n]^{d-2}} \int_{\gamma_{x, v_2, v_3, \ldots, v_{d-1}}} \rho(t) \,dt \,d(\mu_1 \times \lambda_{d-2}) \\
&= \int_{I_n \times [0,s_n]^{d-1}} \rho \,d\mu_d.
\end{align*}
Applying H\"older's inequality, we get
\begin{align*}
\int_{I_n \times [0,s_n]^{d-1}} \rho \,d\mu_d
&\leq \br{ \int_{I_n \times [0,s_n]^{d-1}} \,d\mu_d}^{\frac{1}{\delta}} \br{ \int_{I_n \times [0,s_n]^{d-1}} \rho^d \,d\mu_d}^{\frac{1}{d}}
\end{align*}
where $\delta$ is the conjugate exponent of $d$ (so $\frac{1}{\delta} + \frac{1}{d} = 1$). We have 
\begin{align*}
\int_{I_n \times [0,s_n]^{d-1}} \rho^d \,d\mu_d
&\geq \br{ \int_{I_n \times [0,s_n]^{d-1}} \,d\mu_d}^{-\frac{d}{\delta}}\br{\int_{I_n \times [0,s_n]^{d-2}} \,d\mu_1 \times \lambda_{d-2}}^d  \\
&=\br{ (s_n)^{(d-1)} \ell(I_n)}^{-\frac{d}{\delta}} \br{(s_n)^{d-2} \ell(I_n)}^d \\
&= 2^{n} \ell(I_n).
\end{align*}
This is true for all admissible function $\rho$, therefore
\[ \mod_d(\Gamma_n) \geq  s_n^{-1} \ell(I_n). \]
As $n \to \infty$, $s_n^{-1} \ell(I_n) = L(\alpha_n, c_n) \to +\infty$. We have 
\[ \mod_d(E_n, F_n) \geq \mod_d(\Gamma_n) \to +\infty. \]
\end{proof}

\begin{prop}\label{HD_Prop_2}
$X_d$ is not quasisymmetrically equivalent to any subset of $\R^d$. 
\end{prop}
\begin{proof}
For each $n \in \N$, and for $x \in I_n, (v_2, \ldots, v_{d-1}) \in [0,s_n]^{d-2}\}$, let $\gamma_{x, v_2, \ldots, v_{d-1}}$ be the path
\[ t \mapsto \br{x, t, v_2, \ldots, v_{d-1}}, t \in [0,s_n]. \]
Let
\[ \Gamma_n = \{ \gamma_{x, v_2, \ldots, v_{d-1}} : x \in I_n, (v_2, \ldots, v_{d-1}) \in [0,s_n]^{d-2}\}\]
be the family of straight lines joining $E_n$ and $F_n$ that meet $E_n$ orthogonally.
Suppose $f: X_d \to \R^d$ is a $\eta$-quasisymmetric embedding. We denote by $\gamma^*$ the image of the curve $\gamma$. For a path family $\Gamma$, define $\Gamma^* = \bigcup_{\gamma \in \Gamma} \gamma*$, and define $\diam(\Gamma) = \inf\{\diam \gamma: \gamma \in \Gamma\}$.

Choose $E_n$ and $F_n$ as in Proposition \ref{HD_Prop_3}. We know from $\diam(E_n) \wedge \diam(F_n) = d(E_n, F_n)$ that 
\[ \frac{d(f(E_n), f(F_n))}{\diam(f(E_n)) \wedge \diam(f(F_n))} \sim 1, \]
where the implicit constant for $\sim$ depends only on $\eta$. For the same reason, there exists $\alpha > 1$, depending only on $\alpha,\beta > 1$, depending only on $\eta$, so that $f(\Gamma_n)^* \subset B_n$ for a ball $B_n$ with diameter $r_n \leq \alpha \,d(E_n, F_n) \leq \beta \diam(f(\Gamma))$. 

For each $n \in \N$, we can cover $\Gamma^*$ by squares $\{R_i\}_{i \in I_n}$ of diameter $s_n\varphi_n(c_n)$ so that their sides are either parallel to the paths in $\Gamma$ or orthogonal to the paths in $\Gamma$. We can choose $\{R_i\}$ so that these $R_i$'s don't overlap and their union is precisely $\Gamma^*$. For each rectifiable path $\gamma$, denote by $\ell(\gamma)$ its length. 
Let 
\[ \rho_n = (\diam f(\Gamma_n))^{-1} \sum_{i \in I_n} \frac{\diam(f(R_i))}{\diam(R_i)} \mathbbm{1}_{f^{-1}(B_n) \cap R_i}. \]
be a function on $X_d$. 
For all $\gamma \in \Gamma_n$, 

 We have 
\[ \int_{\gamma} \rho_n(s) \,ds 
= (\diam f(\Gamma_n))^{-1} \sum_{i \in I_n} \frac{\diam(f(R_i))}{\diam(R_i)} \ell( f^{-1}(B_n) \cap R_i \cap \gamma). \]
For each $i$, $f^{-1}(B_n) \cap R_i \cap \gamma = R_i \cap \gamma$. When $R_i \cap \gamma \neq \emptyset$, $\ell(R_i \cap \gamma) = \diam(R_i)$. As $\{R_i\}_{i \in I_n}$ covers $\Gamma^*$, we have
\begin{align*}
 \int_{\gamma} \rho_n(s) \,ds 
\geq (\diam f(\Gamma_n))^{-1}  \sum_{i \in I_n, R_i \cap \gamma \neq \emptyset } \diam(f(R_i))
\geq  (\diam f(\Gamma_n))^{-1}  \diam (f(\gamma))
\geq 1.
\end{align*}
This means $\rho_n$ is admissible for $\Gamma_n$. We have 
\begin{align*}
\mod_d \Gamma_n
&\leq \int \rho_n^d \,d\mu_d 
= (\diam f(\Gamma_n))^{-d} \sum_{i \in I_n} \br{ \frac{\diam(f(R_i))}{\diam(R_i)}}^d \mu_d(f^{-1}(B_n) \cap R_i) \\
&= (\diam f(\Gamma_n))^{-d} \sum_{i \in I_n} \br{ \frac{\diam(f(R_i))}{\diam(R_i)}}^d \mu_d(R_i) \\
&= (\diam f(\Gamma_n))^{-d} \sum_{i \in I_n} \br{ \frac{\diam(f(R_i))}{\diam(R_i)}}^d \diam(R_i)^d.
\end{align*}
For the last equality, we make use of the fact that $R_i$ are chosen so small that $\mu_d(R_i) = \diam(R_i)^d$. We get
\begin{align*}
\mod_d \Gamma_n
\lesssim (\diam f(\Gamma_n))^{-d} \sum_{i \in I_n} \br{ \diam(f(R_i)) }^d
\lesssim (\diam f(\Gamma_n))^{-d} \sum_{i \in I_n} \lambda_d(f(R_i)).
\end{align*}
Here we use the fact that the Lebesgue $\lambda_d$ on $\R^d$ is $d$-Ahlfors regular and that $f(R_i)$ are uniform quaisdisks. Since $\{f(R_i)\}_{n\in\N}$ are disjoint subsets of $B_n$, we have 
\begin{align*}
\mod_d \Gamma_n \lesssim (\diam f(\Gamma_n))^{-d} \lambda_d(B_n) \lesssim 1.
\end{align*}
where all the implicit constants for $\lesssim$ depends only on $\eta$. But this is a contradiction to Proposition \ref{HD_Prop_3}. 
\end{proof}

From Prop \ref{HD_Prop_1} and Prop \ref{HD_Prop_2}, $X_d$ is homeomorphic to $\R^d$, doubling and linearly localy connected, it is not quasisymmetric to $\R^d$. Our proof shows that the any ball in $X_d$ centered at $0$ cannot be quasisymmetrically embedded into $\R^d$. 

We conclud this section by a proof of Theorem \ref{Theorem}. 
\begin{proof}[Proof of Thoerem \ref{Theorem}]
The $d$-dimensional unit ball $B(0,1)$ in $\R^d$, equipped with the metric 
\[ \wtil{\rho}(x, y) = \frac{\rho\br{ \frac{x}{1-\abs{x}}, \frac{y}{1-\abs{y}}}}{1+ \rho\br{ \frac{x}{1-\abs{x}}, \frac{y}{1-\abs{y}}}}. \]
The completion of the space $(B(0,1), \wtil{\rho})$ is $\wbar{B}(0,1)$, and the metric on the boundary is same as the Euclidean metric. Glue the space $\wbar{B}(0,1)$ with another hemisphere to form a topological $d$-sphere. This $d$-sphere is doubling, locally linearly contractible, and every weak tangent is isometric to $\R^d$, but it cannot be a quasisphere.
\end{proof}

\section{Ahlfors Regularity}
Let $Q > 0$. A measured metric space $(X, d, \mu)$ is said to be $Q$-Ahlfors regular if for all $x \in X$ and $r \leq \diam X$, we have
\[ \mu(B(x,r)) \sim r^Q. \]
Ahlfors regularity and doubling property of metric spaces are related notions. We record a result from \cite{Heinonen}:
\begin{thm} \label{Heinonen_AR_2}\cite[Theorem 14.6]{Heinonen}
Let $X$ be a complete, connected metric space of finite Assouad dimension $\beta$. Then for each $Q > \beta$, there exists a quasisymmetric homeomorphism of $X$ onto a closed $Q$-Ahlfors regular subset of some $\R^N$.
\end{thm}
With these facts our example gives:
\begin{thm}\label{Q-AR-counter-example}
For every $Q > 2$, there exists a $Q$-Ahlfors regular and linearly locally contractible metric space $X$ that is topologically an $2$-sphere such that every weak tangent is uniformly quasisymmetric to $\R^2$ but $X$ is not quasisymmetric to the standard $2$-sphere.
\end{thm}
\begin{proof}
Let $Q > 2$. By Proposition \ref{Prop_4}, the Assouad dimension of $X_2$ is $2$. Proposition \ref{Heinonen_AR_2} says that there exist a distortion function $\eta:[0,\infty) \to [0,\infty)$ and a $\eta$-quasisymmetry $\varphi:X_2 \to X'$, where $X'$ is a closed $Q$-Ahlfors regular subset of $\R^N$. By Proposition \ref{HD_Prop_2} and Proposition \ref{HD_Prop_1}(b), $X'$ is not quasisymmetric to the standard $2$-sphere, but every weak tangent of $X'$ is $\eta$-quasisymmetric to $(\R^2, 0)$.
\end{proof}

\bibliographystyle{alpha}
\bibliography{/home/angela/Desktop/Dropbox/Research/Recurrence_Self_Similar_Graphs/Biblography.bib}

\end{document}